\documentclass[11pt]{article} 
\usepackage{amsmath,amsthm} 
\usepackage{amssymb,latexsym}
\usepackage[T1]{fontenc}
\usepackage{authblk}
\usepackage{lipsum}
\usepackage{hyperref}
\usepackage{lipsum} 
\usepackage{lineno}
\usepackage{upgreek}
\usepackage{tikz}
\usetikzlibrary{shapes.geometric}
\usetikzlibrary{positioning}
\usetikzlibrary{arrows.meta}
\modulolinenumbers[1]
\begin{document}
	\newtheorem{theorem}{Theorem}[section]
	\newtheorem{question}{Question}[section]
	\newtheorem{thm}[theorem]{Theorem}
	\newtheorem{lem}[theorem]{Lemma}
	\newtheorem{eg}[theorem]{Example}
	\newtheorem{prop}[theorem]{Proposition}
	\newtheorem{cor}[theorem]{Corollary}
	\newtheorem{rem}[theorem]{Remark}
	\newtheorem{deff}[theorem]{Definition}
	\numberwithin{equation}{section}
	
	\title{On $Z$-Noetherian Rings}
	
	\author[a]{Ambar Dubey}
	
	\author[b]{Tushar Singh}
	
	\author[a]{Shiv Datt Kumar}

	\affil[a]{\small Department of Mathematics, Motilal Nehru National Institute of Technology Allahabad, Prayagraj 211004, India \vskip0.01in Emails: iiitabr11196@gmail.com, sdt@mnnit.ac.in}
	\vskip0.05in
	\affil[b]{\small Department of Applied Sciences (Mathematics), National Institute of Technology Delhi, Plot No. FA7, Zone P1, GT Karnal Road, Delhi-110036, India \vskip0.01in Emails: sjstusharsingh0019@gmail.com}
	\maketitle
	\hrule
	
	\begin{abstract}
		In this paper, introduce a new finiteness condition for commutative rings, called $Z$-Noetherian, under which every ideal contained in the set of zero-divisors $Z(R)$ is finitely generated. This condition lies properly between Noetherian and $Nil_{*}$-Noetherian. We establish intrinsic characterizations of $Z$-Noetherian rings and investigate their stability under nilpotent quotients and localizations. Under the assumption that $Z(R)$ is an ideal, we obtain a primary decomposition theorem for ideals contained in $Z(R)$. For $ZD$-rings, we prove a Cohen-type criterion characterizing $Z$-Noetherian in terms of the finite generation of prime ideals contained in $Z(R)$. We further establish $Z$-Noetherian analogues of the Hilbert Basis Theorem and Krull's Intersection Theorem. Finally, we establish a principal ideal theorem for $Z$-Noetherian rings.  Collectively, these results show that several fundamental finiteness and  structural properties of Noetherian rings extend to the broader class of  $Z$-Noetherian rings.
	\end{abstract}
	
	\hrule
	\noindent
	\section{Introduction}
	\noindent
	Throughout this paper, all rings are assumed to be commutative with identity, and all modules are unitary. We denote by $Z(R)$ the set of zero-divisors of $R$, Nil$(R)$ the nilradical of $R$, and  $R[X]$ the polynomial ring over $R$. Noetherian rings \cite{ne21} play a fundamental role in commutative algebra, primarily because of their underlying finiteness condition that every ideal is finitely generated. This finiteness property is essential in many important developments of the subject and provides the framework for several classical results, including Hilbert's Basis Theorem, the Lasker-Noether Theorem on primary decomposition, Krull's Intersection Theorem, Cohen's Theorem, and Krull's Principal Ideal Theorem. A natural direction in commutative algebra is to investigate whether these finiteness properties and their associated structural results remain valid when the Noetherian condition is suitably weakened. This has led to the introduction and study of various classes of generalized Noetherian rings, including Nonnil-Noetherian rings \cite{ab03}, $S$-Noetherian rings \cite{ad02}, $J$-Noetherian rings \cite{ka24}, $S$-$J$-Noetherian rings \cite{ts26}, $Nil_{*}$-Noetherian rings \cite{zx22}, and, more recently, $Nil_{*}$-$S$-Noetherian rings \cite{ts25}. For a detailed account of these and other generalizations of Noetherian rings, we refer the reader to the survey article \cite{ts23}. These generalized classes provide broader settings in which various finiteness conditions and classical results of Noetherian rings can be investigated and extended. A ring $R$ is called Nil$_{*}$-Noetherian if every nil ideal of $R$ is finitely generated (see \cite[Definition 1.1]{zx22}. Although this class properly contains Noetherian rings, many important finiteness properties continue to hold. Since every nil ideal is contained in the set of zero-divisors $Z(R)$, it is natural to ask whether requiring finite generation only for ideals contained in $Z(R)$ is sufficient to recover further structural properties of Noetherian rings. Motivated by this observation, we introduce the notion of a $Z$-\textit{Noetherian ring}. A commutative ring $R$ is said to be $Z$-Noetherian if every ideal contained in $Z(R)$ is finitely generated (see Definition \ref{z1}). Clearly,
	\[
	\text{Noetherian} \Longrightarrow \text{$Z$-Noetherian} \Longrightarrow \text{Nil}^{*}\text{-Noetherian},
	\]
	and neither implication is reversible (see Example \ref{z2}) . Thus, $Z$-Noetherian rings constitute a natural intermediate class between Noetherian rings and $Nil_{*}$-Noetherian rings. The main objective of this paper is to investigate how far the classical theory of Noetherian rings extends to the setting of $Z$-Noetherian rings. We establish several equivalent characterizations of $Z$-Noetherian rings and investigate their behavior under quotients and localizations, along with various permanence properties. Furthermore, we develop a primary decomposition theory for ideals contained in $Z(R)$ (see Theorem \ref{sp}) and establish a Cohen-type characterization for $Z$-Noetherian $ZD$-rings (see Theorem \ref{z3}). Furthermore, we prove analogues of several classical theorems, including Hilbert's Basis Theorem (see Theorem \ref{z4}), Krull's Intersection Theorem (see Theorem \ref{z5}), and Krull's Principal Ideal Theorem (see Theorem \ref{z6}), under suitable assumptions. These results demonstrate that many fundamental finiteness properties of commutative algebra continue to hold when finite generation is assumed only for ideals contained in the set of zero-divisors.
	
	\section{Main Results}
	We begin by introducing the following new class of Noetherian rings.
	\begin{deff}\label{z1}
		A ring $R$ is called $Z$-Noetherian if every ideal $I\subseteq Z(R)$ is finitely generated.
	\end{deff}
	\noindent
	\begin{rem}
		When $Z(R)=Nil(R)$, the notions of $Z$-Noetherian and Nil$_{*}$-Noetherian ring same. Obviously, $Nil(R)\subseteq Z(R)$ but $Z(R)\nsubseteq Nil(R)$ (in general) for this case $Z$-Noetherian ring is different class of ring which lies between Noetherian and  Nil$_{*}$-Noetherian ring class. For more details, we illustrate an example of ring which is $Nil_{*}$-Noetherian ring but not a $Z$-Noetherian as follows:
	\end{rem}
	\begin{eg}\label{z2}
		Let $R=\prod\limits_{i\in \mathbb{N}} \mathbb{Z}_{2}$. Then $Z(R)=\{(a_i)\in R \mid \text{there exists} i \text{ such that} a_i=0\}$, and Nil$(R)=(0)$. Thus $R$ is a Nil$_{*}$-Noetherian ring. Let $I=\{(a_1,a_2,\dots)\in R \mid a_i=0  ~for ~all ~but ~finitely ~many~  i\}$. Then $I$ is an ideal of $R$ with $I\subseteq Z(R)$, but $I$ is not finitely generated. Hence $R$ is not a $Z$-Noetherian ring.
	\end{eg}
	We next investigate the behavior of $Z$-Noetherian rings under quotients.
	\begin{prop}
		Let $R$ be a $Z$-Noetherian ring and let $I$ be a nil ideal of $R$. Then $R/I$ is $Z$-Noetherian.
	\end{prop}
	\begin{proof}
		Let $\varphi : R \longrightarrow R/I$ be the natural ring homomorphism defined by $\varphi(a)=a+I$. Let $J$ be an ideal of $R/I$ such that $J \subseteq Z(R/I)$. Set
		$J'=\varphi^{-1}(J)$. Then $J'$ is an ideal of $R$. We claim that $J' \subseteq Z(R)$. Indeed, let $a\in J'$. Then $a+I\in J\subseteq Z(R/I)$, so there exists $b+I\in R\setminus I$ such that $(a+I)(b+I)=I$. Hence $ab\in I$.
		Since $I$ is a nil ideal, there exists $n\ge 1$ such that
		$(ab)^n=0$. Since $b\notin I$, and hence $b^n\neq 0$ for any $n\in\mathbb{N}$. Therefore $a^n
		\in  Z(R)$, which implies that $a\in Z(R)$. Thus $J'\subseteq Z(R)$. Since $R$ is a $Z$-Noetherian, $J'$ is finitely generated. Finally, since
		$J=\varphi(J')$, it follows that $J$ is finitely generated.
		Therefore $R/I$ is $Z$--Noetherian.
	\end{proof}

	\begin{prop} \label{n}
		Let $R$ be a Noetherian ring, and let $I$ be an ideal of $R$ such that $I \subseteq Z(R)$. Then $\operatorname{Ann}_R(I) \neq (0)$.
	\end{prop}
	\begin{proof}
		\noindent
		Suppose $R$ is Noetherian, the set of zero divisors of $R$ is the union of the associated primes of the $R$-module $R$: $Z(R) = \bigcup_{\mathfrak{p} \in \operatorname{Ass}_R(R)} \mathfrak{p}$, by \cite[Theorem 4.7]{fm69}. Also $\operatorname{Ass}_R(R)$ is finite, and each $\mathfrak{p} \in \operatorname{Ass}_R(R)$ \[
		\mathfrak{p} = \operatorname{Ann}(m) = \{ r \in R : rm = 0 \}, \quad m \in R,\ m \neq 0\]
		Since $I\subseteq Z(R)$, we have $I\subseteq \bigcup_{\mathfrak{p} \in \operatorname{Ass}_R(R)} \mathfrak{p}$,  by \cite[Proposition 1.11]{fm69}, $I\subseteq \mathfrak{p}$ for some $\mathfrak{p}\in \operatorname{Ass}_R(R)$. Since $I\subseteq \mathfrak{p}$ so
		$am = 0$ for all $a \in I$. Hence $m\in \operatorname{Ann}_R(I)$ and $m \neq 0$, so $\operatorname{Ann}_R(I) \neq(0) $.
	\end{proof}
	\noindent
	\textbf{Counterexample of \ref{n}  (In the non-Noetherian case).}
	Let $R = k[x_1, x_2, x_3, \dots]/(x_1^2, x_2^2, x_3^2, \dots)$, where $k$ is a field.  Let $I = (x_1, x_2, x_3, \dots)$, which is not finitely generated. Clearly, $I \subseteq Z(R)$. We show that $\operatorname{Ann}_R(I) = (0)$. Let $r \in R$ such that $rI = 0$. Then $r x_i = 0$ for all $i \ge 1$. Since $R$ is a quotient of a polynomial ring by $(x_i^2)$, every element $r \in R$ can be written as a finite $k$-linear combination of square-free monomials:\[
	r x_i = a_0 x_i + \sum_{\substack{j \\ x_i \notin \mathfrak{m}_j}} a_j \mathfrak{m}_j x_i
	\] and each $m_j$ is a product of distinct variables. Fix $i$. Then multiplying by $x_i$, all monomials $m_j$ not involving $x_i$ give nonzero terms $m_j x_i$, while those involving $x_i$ vanish . Thus, $r x_i = a_0 x_i + \sum_{\substack{j \\ x_i \notin m_j}} a_j\, m_j x_i$. Since $r x_i = 0$ and the monomials in $R$ are linearly independent over $k$ therefore  $a_0 = 0$ and  $a_j = 0$. Hence $r = 0$, and therefore $\operatorname{Ann}_R(I) = (0)$. This shows that the conclusion of the lemma fails in the non-Noetherian case.\\
	
	Next, we establish several equivalent characterizations of $Z$-Noetherian rings, analogous to the classical characterizations of Noetherian rings.
	
	\begin{theorem}\label{j1}
		Let $R$ be a commutative ring. The following are equivalent:
		\begin{enumerate}
			\item $R$ is $Z$-Noetherian.
			
			\item Every ideal $I\subseteq Z(R)$ is finitely generated.
			
			\item The set of ideals contained in $Z(R)$ satisfies the ascending chain condition.
			
			\item Every nonempty set of ideals contained in $Z(R)$ has a maximal element.
		\end{enumerate}
	\end{theorem}
	\begin{proof}
		\noindent
		$(1)\Leftrightarrow(2)$ follows directly from the definition.\\
		$(2)\Rightarrow(3)$. Let $I_1 \subseteq I_2 \subseteq I_3 \subseteq \cdots$ be an ascending chain of ideals of $R$ such each $I_{i\in\mathbf{N}}$ contained in $Z(R)$. Set $I=\bigcup\limits_{j\ge 1} I_j$. Clearly $I\subseteq Z(R)$ and $I$ is an ideal of $R$. By assumption, $I$ is finitely generated, say $I=(a_1,a_2,\dots,a_m)$. Each generator $a_r$ belongs to some $I_{N_r}$. Let $N=\max\{N_1,\dots,N_m\}$.
		Then $a_1,\dots,a_m\in I_N$, hence $I\subseteq I_N$.
		Since always $I_N\subseteq I$, we obtain $I=I_N$. Therefore the chain stabilizes, and the ascending chain condition holds.\\
		$(3)\Rightarrow(4)$: Assume that ACC holds for ideals contained in $Z(R)$. Let $\mathcal{F}$ be a nonempty family of ideals with each $I\in\mathcal{F}$ satisfying $I\subseteq Z(R)$. Suppose that $\mathcal{F}$ has no maximal element.Then for every $I_j\in\mathcal{F}$ there exist $I_{j+1}\in\mathcal{F}$ with $I_j \subsetneq I_{j+1}$. Then we can construct a strictly increasing chain $I_1 \subsetneq I_2 \subsetneq I_3 \subsetneq \cdots$ with each $I_j\in\mathcal{F}$, contradicting ACC. Hence $\mathcal{F}$ must have a maximal element.\\
		$(4)\Rightarrow(3)$: 
		Assume that every nonempty family of ideals contained in $Z(R)$ has a maximal element. If ACC failed, then there would exist a strictly increasing chain $I_1 \subsetneq I_2 \subsetneq I_3 \subsetneq \cdots$ of ideals contained in $Z(R)$. Then the family
		$\{ I_n \mid n\in\mathbb{N} \}$ would have no maximal element, contradicting (4). Therefore ACC holds for ideals contained in $Z(R)$.\\
		$(3)\Rightarrow(2)$: Assume that ACC holds for ideals contained in $Z(R)$. Let $I\subseteq Z(R)$ be an ideal. Suppose that $I$ is not finitely generated. Then we can choose elements $a_1,a_2,a_3,\dots \in I$ such that $(a_1)\subsetneq (a_1,a_2)\subsetneq (a_1,a_2,a_3)\subsetneq \cdots$. This gives a strictly ascending chain of ideals contained in $Z(R)$, contradicting ACC. Therefore $I$ must be finitely generated.
	\end{proof}
	\noindent
	Recall from \cite{fm69} that a proper ideal $I$ of a ring $R$ is said to be \textit{primary} if, whenever $ab \in I$, either $a \in I$ or $b^n \in I$ for some $n > 0$. An ideal $I$ of $R$ is said to be \textit{irreducible} if
	$I = K \cap L \implies I = K$ or $I = L$ for some ideal $K, L$ of $R$. We next extend the primary decomposition theorem to the setting of $Z$-Noetherian rings. We begin with the following lemma.
	
	\begin{lem}\label{j2}
		Let $R$ be a $Z$ Noetherian ring, and let $I$ be an ideal of $R$ such that $I\subseteq Z(R)$. If $I$ is an irreducible ideal of $R$, then $I$ is a primary ideal.
	\end{lem}
	\begin{proof}
		\noindent
		To prove that $I$ is primary, we must show that if $ab \in I$ and $b \notin I$, then there exists an integer $n \geq 1$ such that $a^n \in I$. Consider the colon ideal $(I : a^n) = \{\, r \in R \mid ra^{n}\in I \,\}$. Since $R$ is a $Z$-Noetherian ring, any ascending chain of ideals contained in $Z(R)$ stabilizes, by Theorem \ref{j1}. We divide the proof into two cases.
		\medskip
		
		\noindent\textbf{Case 1:} $(I : a^n) \subseteq Z(R)$ for all $n \geq 1$. Consider the ascending chain $(I : a) \subseteq (I : a^2) \subseteq (I : a^3) \subseteq \cdots$. By assumption, all these ideals are contained in $Z(R)$, so the $Z$-Noetherian property implies that the chain stabilizes, by Theorem \ref{j1}. Hence, there exists $N \geq 1$ such that $(I : a^N) = (I : a^{N+1})$. We claim that $I = (I + (a^n)) \cap (I + (b))$. The inclusion $I \subseteq (I + (a^n)) \cap (I + (b))$ is clear. Conversely, let $x \in (I + (a^n)) \cap (I + (b))$. Then $x = i_1 + r a^n= i_2 + s b$ for some $i_1, i_2 \in I$ and $r, s \in R$. Multiplying by $a$, we get $ax = a i_1 + r a^{n+1} = a i_2 + s ab$. Since $i_1, i_2 \in I$ and $ab \in I$, the right-hand side lies in $I$. Hence $ax \in I$. Thus $a i_1 + r a^{n+1} \in I$. Since $a i_1 \in I$, it follows that $r a^{n+1} \in I$, so $r \in (I : a^{n+1}) = (I : a^n)$. Hence $ra^n \in I$, and therefore $x = i_1 + r a^n \in I$. Thus $(I + (a^n)) \cap (I + (b)) \subseteq I$, and equality holds. Since $I$ is irreducible, it follows that either $I = I + (a^n)$ or $I = I + (b)$. As $b \notin I$, we must have $I = I + (a^n)$, which implies $a^n \in I$, as desired.
		\medskip
		
		\noindent \textbf{Case 2:} $(I : a^n) \nsubseteq Z(R)$ for some $n \geq 1$. Then there exists $y \in (I : a^n)$ such that $y \notin Z(R)$. Hence $y$ is a regular element. Since $y \in (I : a^n)$, we have $y a^n \in I$. Consider the chain $(I : y) \subseteq (I : y^2) \subseteq (I : y^3) \subseteq \cdots$. Let $z \in (I : y^k)$. Then $z y^k \in I \subseteq Z(R)$. Since $y$ is regular, so is $y^k$, and hence $z$ must be a zero-divisor. Thus $(I : y^k) \subseteq Z(R)$ for all $k \geq 1$. By the $Z$-Noetherian property, the chain stabilizes, by Theorem \ref{j1}. So there exists $l\geq 1$ such that $(I : y^l) = (I : y^{l+1})$. We claim that $I = (I + (y^l)) \cap (I + (a^n))$. Let $x \in (I + (y^l)) \cap (I + (a^n))$. Then $x = j_1 + u y^l= j_2 + v a^n$ for some $j_1, j_2 \in I$ and $u, v \in R$. Multiplying by $y$, we get $yx = y j_1 + u y^{l+1} = y j_2 + v y a^n$. Since $y a^n \in I$ and $j_2 \in I$, the right-hand side lies in $I$, so $yx \in I$. Thus $y j_1 + u y^{l+1} \in I$. Since $y j_1 \in I$, we get $u y^{l+1} \in I$, hence $u \in (I : y^{l+1}) = (I : y^l)$. Therefore $u y^l \in I$, and so $x = j_1 + u y^l\in I$. Thus $(I + (y^M)) \cap (I + (a^n)) = I$. Since $y$ is regular, $y^l$ is also regular, so $y^M \notin I$ (because $I \subseteq Z(R)$). Hence $I \neq I + (y^l)$. Since $I$ is irreducible, it follows that $I = I + (a^n)$, so $a^n \in I$. Therefore, in both cases, $a^n \in I$ for some $n \geq 1$, and hence $I$ is primary.
	\end{proof}
	
	We are now in a position to establish the following primary decomposition theorem for $Z$-Noetherian rings.
	
	\begin{theorem}\label{sp}
		\textbf{(Existence of primary decomposition in $Z$-Noetherian rings)}
		Let $R$ be a $Z$-Noetherian ring such that $Z(R)$ is an ideal of $R$. Then every proper ideal $I$ of $R$ contained in $Z(R)$ can be expressed as a finite intersection of primary ideals.
	\end{theorem}
	
	\begin{proof}
		\noindent
		Let $E= \{\,I \subseteq Z(R) \mid I~is~not~ a~ finite~ intersection~of~ primary~ ideals \,\}$. 
		We claim that $E = \emptyset$. Suppose, to the contrary, that $E \neq \emptyset$. Since $R$ is a $Z$-Noetherian ring, by Theorem~\ref{j1}, $E$ contains a maximal element with respect to inclusion. Let $I\in E$ be such a maximal element. Clearly, $I$ is not a primary ideal. By Lemma~\ref{j2}, $I$ is not irreducible. Hence there exist proper ideals $K$ and $L$ of $R$ such that $I = K \cap L$, with $I \subsetneq K$ and $I \subsetneq L$. Set $K_1 = K \cap Z(R)$ and $L_1= L \cap Z(R)$. Since $I \subseteq Z(R)$, $K_1 \cap L_1 = (K \cap Z(R)) \cap (L \cap Z(R)) = (K \cap L) \cap Z(R) = I \cap Z(R) = I$, so $I = K_1 \cap L_1$, with $K_1, L_1 \subseteq Z(R)$. We claim that $K_1, L_1 \notin E$. Suppose, for instance, that $K_1 \in E$. Since $I \subsetneq K_1$ and $I$ is maximal in $E$, it follows that $K_1 = I$, a contradiction. Similarly, $L_1 \notin E$. Then  $L_1$ and $K_1$ is a finite intersection of primary ideals. Since $I = K_1 \cap L_1$, it follows that $I$ is also a finite intersection of primary ideals, which contradicts the fact that $I \in E$. Therefore, $E = \emptyset$. Hence, every ideal $I \subseteq Z(R)$ can be written as a finite intersection of primary ideals. 
	\end{proof}
	
	\begin{eg}
		Let $k$ be a field and let $R = k[x,y]/(x^2,xy)$. Since $\bar{x}^2 = 0$, the element $\bar{x}$ is nilpotent, and  $\text{Nil}(R) = (\bar{x})$. Here $Z(R) = \{f(\bar y) + a\bar{x} \mid f(0) = 0\}$. Moreover, $Z(R) = (\bar{x},\bar{y})$ is an ideal of $R$, and $Z(R) \neq \text{Nil}(R)$. 
	\end{eg} 
	
	Recall that a multiplicatively closed subset of $R$ is a subset $S$ of $R$ such that $1\in S$ and $S$ is closed under multiplication. We next investigate the behavior of $Z$-Noetherian rings under localization and establish that this property is preserved under suitable conditions.
	
	\begin{theorem}
		Let $R$ be a $Z$-Noetherian ring and let $S\subseteq R$ be a multiplicatively closed set such that $S\cap Z(R)=\emptyset$. Then $S^{-1}R$ is $Z$-Noetherian.
	\end{theorem}
	\begin{proof}
		\noindent
		Assume that $R$ is $Z$-Noetherian and $S\subseteq R$ is multiplicatively closed with $S\cap Z(R)=\varnothing$. To show that $S^{-1}R$ is $Z$--Noetherian, it suffices to prove that every ideal of $S^{-1}R$ contained in $Z(S^{-1}R)$ is finitely generated. We claim that $Z(S^{-1}R)=S^{-1}Z(R)$. 
		Let $\frac{a}{s}\in S^{-1}Z(R)$. Then $a\in Z(R)$, so there exists $b\neq 0$ in $R$ with $ab=0$. Hence $\frac{a}{s}\cdot\frac{b}{1}=0$ in $S^{-1}R$, and therefore
		$\frac{a}{s}\in Z(S^{-1}R)$. Thus $S^{-1}Z(R)\subseteq Z(S^{-1}R)$. Conversely, let $\frac{a}{s}\in Z(S^{-1}R)$. Then there exists $\frac{b}{t}\neq 0$ in $S^{-1}R$
		such that $\frac{a}{s}\cdot\frac{b}{t}=0$, and hence $uab=0$ for some $u\in S$. Since $S\cap Z(R)=\varnothing$, $u$ is not a zero divisor, so $ab=0$, which implies $a\in Z(R)$. Thus $\frac{a}{s}\in S^{-1}Z(R)$, and the equality follows. Now let $J$ be an ideal of $S^{-1}R$ such that $J\subseteq Z(S^{-1}R)$. Define the contraction $I=\{\,a\in R\mid \frac{a}{1}\in J\,\}$. Then $I$ is an ideal of $R$ and $J=S^{-1}I$. Since $J\subseteq S^{-1}Z(R)$, we have $I\subseteq Z(R)$. As $R$ is $Z$--Noetherian, the ideal $I$ is finitely generated, say $I=(a_1,\dots,a_n)$. Hence $J=S^{-1}I=(\frac{a_1}{1},\dots,\frac{a_n}{1})$, and therefore, it $J$ is finitely generated. Thus every ideal $S^{-1}R$ contained in $Z(S^{-1}R)$ is finitely generated and consequently $S^{-1}R$ is $Z$--Noetherian.
	\end{proof}
	
	\begin{deff}\cite{eg71}
		A ring $R$ is called a \textit{ZD-ring} if, for every ideal $I$ of $R$, the set of zero divisors on $R/I$, $Z_R(R/I)=\{\, r\in R \mid r(a+I)=0 \text{ for some } a\in R\setminus I \,\}$ is a finite union of prime ideals of $R$.
	\end{deff}
	
	Next, we establish a Cohen-like theorem for $Z$-Noetherian rings as follows:
	
	\begin{theorem}[\textbf{Cohen-type theorem for Z-Noetherian rings}]\label{z3}
		Let $R$ be a ZD-ring. Then $R$ is $Z$-Noetherian if and only if every prime ideal $P$ of $R$ with $P\subseteq Z(R)$ is finitely generated.
	\end{theorem}
	\begin{proof}
		\noindent
		Assume first that $R$ is $Z$--Noetherian. Then every ideal $I$ of $R$ with $I\subseteq Z(R)$ is finitely generated. Hence, every prime ideal $P\subseteq Z(R)$ is finitely generated. Conversely, suppose that every prime ideal $P$ of $R$ with $P\subseteq Z(R)$ is finitely generated. Assume that $R$ is not $Z$--Noetherian. Let $\mathcal{F}=\{ I\subseteq Z(R)\mid I ~is ~non ~finitely ~generated ~ideal ~of ~R\}$. Then $\mathcal{F}\neq\varnothing$. By Zorn's Lemma, $\mathcal{F}$ has a maximal element $P$.
		Since $R$ is a ZD--ring, the set $Z_R(R/P)$ is a finite union of prime ideals of $R$, say $Z_R(R/P)=\bigcup_{i=1}^n P_i$, where each $P_i$ is a prime ideal of $R$.
		We show that $P\subseteq Z_R(R/P)$. Let $p\in P$. Since $p\in Z(R)$, there exists $r\in R\setminus\{0\}$ such that $pr=0$. If $r\in P$, then $r+P=0$ in $R/P$. Otherwise, $r+P\neq 0$ in $R/P$, and in either case we have $p(r+P)=0$ in $R/P$. Hence $p\in Z_R(R/P)$, and therefore
		$P\subseteq Z_R(R/P)$. Thus $P\subseteq P_i$ for some $i$. By the maximality of $P$ in $\mathcal{F}$, we must have $P=P_i$. Hence $P$ is a prime ideal of $R$ contained in $Z(R)$. By hypothesis, $P$ is finitely generated, which contradicts the choice of $P$. Therefore $R$ is $Z$--Noetherian.
	\end{proof}
	
	\begin{theorem}[\textbf{Krull intersection section type theorem}]\label{z5}
		Let $R$ be a $Z$-Noetherian ring such that $Z(R)$ is an ideal of $R$ and $I$ is an ideal of $R$ such that $rad(I) \subseteq Z(R)$. Let $J=\bigcap_{n=1}^{\infty} I^{n}$. Then there exists an element $r\in R$ with $r\cong 1(mod (I))$ such that $rJ=0$.
	\end{theorem}
	
	\begin{proof}
		\noindent
		First, note that $J = \bigcap_{n=1}^\infty I^n \subseteq I$, and hence $IJ \subseteq I^2 \subseteq I \subseteq Z(R)$. Thus $B = IJ \subseteq Z(R)$. Since $B \subseteq Z(R)$ and $R$ is $Z$-Noetherian, by Theorem 2.9, the ideal $B$ admits a primary decomposition: $B = B_1 \cap B_2 \cap \dots \cap B_k$, where each $B_i$ is $P_i$-primary and $P_i = \text{rad}(B_i)$. Fix $i$ ($1 \le i \le k$). Observe that $IJ = B \subseteq B_i$. As $B_i$ is $P_i$-primary, either $J \subseteq B_i$ or $I \subseteq P_i$.
		
		\begin{itemize}
			\item \textbf{Case 1: $I \subseteq P_i$.} Since $\text{rad}(I) \subseteq Z(R)$ and $R$ is $Z$-Noetherian, the ideal $I$ is contained in $Z(R)$ and is therefore finitely generated. Write $I =(x_1, x_2, \dots, x_l )$. For each generator $x_j \in I$, since $I \subseteq P_i = \text{rad}(B_i)$, there exists an integer $n_j \in \mathbb{N}$ such that $x_j^{n_j} \in B_i$. Set $m = \sum_{j=1}^l (n_j - 1) + 1$. Then, by the Pigeonhole Principle applied to the generators, $I^m \subseteq B_i$. Consequently, $J = \bigcap_{n=1}^\infty I^n \subseteq I^m \subseteq B_i$.
			
			\item \textbf{Case 2: $I \not\subseteq P_i$.} There exists an element $a \in I \setminus P_i$. Since $aJ \subseteq IJ = B \subseteq B_i$
			and $B_i$ is $P_i$-primary, by the definition of primary ideals, we have $J \subseteq B_i$.
		\end{itemize}
		Combining both cases, we conclude that $J \subseteq B_i$ for all $1 \le i \le k$. Hence, $J \subseteq \bigcap_{i=1}^k B_i = B = IJ$. Since it is trivially true that $IJ \subseteq J$, we obtain the equality $J = IJ$. Note that $J$ is finitely generated because $J \subseteq Z(R)$ and $R$ is a $Z$-Noetherian ring. Finally, \cite[Corollary 2.5]{fm69}, there exists an element $r\in R$ with $r\cong 1(mod (I))$ such that $rJ = 0$. This completes the proof.
	\end{proof}
	Now, we prove the Hilbert basis theorem for $Z$-Noetherian rings. First, we recall McCoy's Theorem as follows:
	\begin{lem}\cite[Theorem 2]{nh42}\label{am1}
		Let $R$ be a commutative ring. An element $f(x)$ of $R[x]$ is a zero divisor in $R[x]$, then there exists a non-zero element $c \in R$ such that $cf(x) = 0$.
	\end{lem}
	
	\begin{theorem}[\textbf{Hilbert Basis Theorem  for $Z$-Noetherian Rings}]\label{z4}
		If $R$ is a $Z$-Noetherian ring, then the polynomial ring $R[x]$ is also $Z$-Noetherian.
		\noindent
	\end{theorem}
	\begin{proof}
		Let $J$ be an ideal of $R[x]$ such that $J \subseteq Z(R[x])$. We show that $J$ is finitely generated.
		For each $i \ge 0$, define $L_i$ to be the set consisting of $0$ and the leading coefficients of all polynomials in $J$ of degree $i$. Then each $L_i$ is an ideal of $R$. Moreover, since $f(x) \in J$ implies $x f(x) \in J$, we obtain an ascending chain
		\begin{center}
			$
			L_0 \subseteq L_1 \subseteq L_2 \subseteq \cdots.
			$
		\end{center}
		Let $L = \bigcup_{i=0}^{\infty} L_i$. Then $L$ is an ideal of $R$.
		First we show that  $L \subseteq Z(R)$.
		Let $a \in L$, $a \neq 0$. Then $a$ is the leading coefficient of some $f(x) \in J$. Since $J \subseteq Z(R[x])$, $f(x)$ is a zero divisor in $R[x]$. By Lemma \ref{am1}, there exists $0 \neq c \in R$ such that $c f(x) = 0$. Hence $c a = 0$, so $a \in Z(R)$. Thus $L \subseteq Z(R)$.
		Since $R$ is $Z$-Noetherian and $L \subseteq Z(R)$, the ideal $L$ is finitely generated.
		Since $L = \bigcup L_i \subseteq Z(R)$, it follows that each $L_i \subseteq Z(R)$. Since $L$ is finitely generated and $L = \bigcup L_i$, all generators lie in some $L_m$, hence: $L = L_m$.
		Hence the chain stabilizes. Then there exists $m$ such that
		$L_m = L_{m+1} = \cdots = L$.
		Each $L_i \subseteq Z(R)$, so each $L_i$ is finitely generated.
		For each $i = 0,1,\dots,m-1$,  construct the  generators. Suppose
		$L_i = (a_{i1}, a_{i2}, \dots, a_{ik_i})$, and choose $f_{ij}(x) \in J$ of degree $i$ with leading coefficient $a_{ij}$.
		Similarly, let $L_m = (a_{m1}, a_{m2}, \dots, a_{mk_m})$, 
		and choose $f_{mj}(x) \in J$ of degree $m$ with leading coefficient $a_{mj}$. Let $J'$ be the ideal generated by all $f_{ij}(x) ,~ 0 \leq i \leq m,\quad 1 \leq j \leq k_i$. Then $J' \subseteq J$.
		
		\noindent Now we prove that $J = J'$ by induction on the degree $d$ that every $f(x) \in J$ lies in $J'$.
		
		\noindent\textit{\textbf{Case 1: $d \ge m$.}}  
		Let $a$ be the leading coefficient of $f(x)$. Then $a \in L_d = L_m$, so $a = \sum_{j=1}^{k_m} r_j a_{mj}$. Define $g(x) = \sum_{j=1}^{k_m} r_j x^{d-m} f_{mj}(x)$. Then $g(x) \in J'$ and $\deg g = d$, and its leading coefficient is $a$. Hence $f(x) - g(x) \in J$ has degree $< d$.
		
		\noindent\textit{\textbf{Case 2:} $d<m$.}  
		Let $a$ be the leading coefficient of $f(x)$. Then $a \in L_d$, so
		$a = \sum_{j=1}^{k_d} r_j a_{dj}$. Define $g(x) = \sum_{j=1}^{k_d} r_j f_{dj}(x)$. Then $g(x) \in J'$, $\deg g = d$, and has leading coefficient $a$. Hence $f(x) - g(x) \in J$ has degree $< d$. In both cases, repeated reduction yields $f(x) \in J'$. Thus $J = J'$. Since $J \subseteq Z(R[x])$ was arbitrary and shown to be finitely generated, it follows that every ideal contained in $Z(R[x])$ is finitely generated. Therefore $R[x]$ is $Z$-Noetherian.
	\end{proof}
	
	\noindent
	Following from \cite{wf16}, \textbf{Krull's Principal Ideal Theorem} states that if $R$ is a Noetherian ring and $a\in R$ is a nonzero nonunit, then every minimal prime ideal $\mathfrak{p}$ over $(a)$ satisfies $\operatorname{ht}(\mathfrak{p})\leq 1$, where $\operatorname{ht}(\mathfrak{p})$ denotes the supremum of the lengths of chains of prime ideals $\mathfrak{p}_0\subsetneq\mathfrak{p}_1\subsetneq\cdots\subsetneq\mathfrak{p}_n=\mathfrak{p}$. Moreover, if $a$ is not a zero-divisor of $R$, then $\operatorname{ht}(\mathfrak{p})=1$. We now establish the corresponding result for $Z$-Noetherian rings.
	
	\begin{theorem}[Krull's Principal Ideal Theorem for $Z$-Noetherian Rings]\label{z6}
		\label{thm:PIT-Znoetherian}
		Let $R$ be a  $Z$-Noetherian ring and  let $x \in R$ be a nonzero nonunit element of $R$. if $P$ is minimal prime over $(x)$ with $P\subseteq Z(R)$. Then
		\begin{enumerate}
			\item[(i)] $R_P$ is Noetherian.\label{hagnu}
			\item[(ii)] $\operatorname{ht}(P) \leq 1$.
		\end{enumerate}
	\end{theorem}
	
	\begin{proof}
		\noindent
		Let $S = R \setminus P$ be a multiplicative closed subset of $R$, and $R_P = S^{-1}R$ is a local ring with unique maximal ideal $PR_P$.
		
		\noindent\textbf{(i)}
		Let $J \subsetneq R_P$ be an ideal and  $J = S^{-1}I$ where  $I = J^{c}$ (the contraction
		of $J$ under the canonical map $R \to R_P$).
		since $J$ is proper, $I \cap S = \emptyset$, by\cite[Corollary 3.11]{fm69} and 
		$I \subseteq P$. By hypothesis $P \subseteq Z(R)$, so $I \subseteq Z(R)$. hence $I = (a_1,\dots,a_n)$ for some $a_1,\dots,a_n \in R$. Consequently, $J$ is finitely generated  hence $R_P$ is Noetherian.
		
		\noindent\textbf{(ii)}
		Since $P$ is minimal over $(x)$ in $R$. Then $PR_P$ is minimal over $(x/1)$ in $R_P$. By\ref{hagnu}, $R_P$ is a Noetherian local ring therefore $\operatorname{ht}(PR_P) \leq 1$, by \cite[Corollary 11.16]{fm69} hence
		$\operatorname{ht}_R(P)\leq 1$, by \cite[Corollary 3.13]{fm69}
	\end{proof}


\end{document}